\documentclass[11pt,reqno]{amsart}

\usepackage{amsfonts, amssymb, amsthm, amsmath, bbm, wasysym, enumerate, amsxtra, latexsym, fullpage, amscd, graphics, epic, youngtab, sansmath, mathtools, multicol, url, hyperref, bbm, blindtext, wasysym, comment, datetime, stmaryrd}
\usepackage[shortlabels]{enumitem}
\usepackage[all,cmtip]{xy}
\usepackage[dvipsnames]{xcolor}
\usepackage[makeroom]{cancel}  
\usepackage{youngtab}
\usepackage{ytableau}
\usepackage{tikz}
\usetikzlibrary{cd}

\usepgflibrary{shapes.geometric}

\tikzstyle{V}=[draw, fill =black, circle, inner sep=0pt, minimum size=1.5pt]
\tikzstyle{C}=[draw, fill =white, circle, inner sep=0pt, minimum size=1.5pt]
\tikzstyle{over}=[draw=white,double=black,line width=2pt, double distance=.5pt]
\numberwithin{equation}{section}
\usepackage[margin = 0.9in]{geometry}
\theoremstyle{definition}

\newtheorem{theorem}{Theorem}[section]
\newtheorem{lemma}[theorem]{Lemma}
\newtheorem{proposition}[theorem]{Proposition}
\newtheorem{corollary}[theorem]{Corollary}

\newtheorem{remark}[theorem]{Remark}

\newtheorem{example}[theorem]{Example}

\numberwithin{equation}{section}

\def\<{\langle}
\def\>{\rangle}

\allowdisplaybreaks 

\tikzstyle directed=[postaction={decorate,decoration={markings,
    mark=at position .65 with {\arrow[arrowstyle]{stealth}}}}]
 \tikzstyle reverse directed=[postaction={decorate,decoration={markings,
    mark=at position .65 with {\arrowreversed[arrowstyle]{stealth};}}}]
    \usetikzlibrary{arrows,shapes,positioning}
\usetikzlibrary{decorations.markings}
\tikzstyle arrowstyle=[scale=1]

\title{M-diagram  basis of the Specht module for $(n,n,n)$}
\author[J. Zhu]{Jieru Zhu}

\address{Institut de Recherche en Mathématique et en Physique,  Universite Catholique de Louvain}
    \email{jieru.zhu@uclouvain.be}

\begin{document}

\begin{abstract}
Motivated by the M-diagrams defined by Tymoczko,  we show that these locally non-crossing $\mathfrak{sl}_3$-webs form a basis of the Specht module for the partition $(n,n,n)$.  They further admit a unitriangular base change to both the polytabloid basis and the non-elliptic web basis.
\end{abstract}
 
\maketitle 

\section{Introduction}
Webs first appear in the study of tensor invariants for Lie algebras $\mathfrak{sl}_2$ \cite{rumer1932valenztheorie} and $\mathfrak{sl}_3$ \cite{Kup96,  KK99}.  In a related manner,  Temperley-Lieb diagrams,  which are $\mathfrak{sl_2}$-webs,  have seen its application in quantum mechanics and theoretical physics.  Recent developments of web categories \cite{CKM14,  RT16,  TVW17} generalize these results to $\mathfrak{gl}_n(\mathbb{C})$ and inspired a new direction in categorification,  link homology and manifold invariants.   It uses planar graphs to describe homomorphism spaces in a full subcategory of $\mathfrak{gl}_n$-modules,  namely,  the one generated by exterior and symmetric powers of the natural module $V$.  Subsequent results also appear for Lie superalgebras of Type Q \cite{BK21,BDK20},  Type P \cite{DKM21a,DKM21b},  and for deformations of $\mathfrak{gl}_n$ over any locally unital superalgebra \cite{DKMZ23}.      

The homomorphism space $\operatorname{Hom}_{\mathfrak{sl}_n(\mathbb{C})}(V^{\otimes d},V^{\otimes t})$ admits a natural action of the symmetric group $\Sigma_d$.    In several cases recorded \cite{PPR09, RT22,IZ19,HK22},  this module coincides with the Specht module for a suitable partition.  In particular,  given a partition $\lambda$ of $d$,  let $S^{\lambda}$ be the irreducible module associated to the cycle type $\lambda$.  It is commonly known as the Specht module and admits a polytabloid basis parametrized by standard Young tableaux of shape $\lambda$.  

The module $\operatorname{Hom}_{\mathfrak{sl}_3(\mathbb{C})}(V^{\otimes 3n},\mathbb{C})$ is isomorphic to $S^{(n,n,n)}$ \cite{PPR09, RT22} and is spanned by non-elliptic webs,  also known as reduced webs,  which have $3n$ boundary points on a horizontal line arranged at the bottom.  The $\Sigma_d$-action is realized as stacking of braid diagrams.  Russell-Tymoczko \cite{RT22} show that the base change from the polytabloid basis to the web basis is unitriangular,  and conjucture that the transitioning matrix has non-negative entries.  A similar version of this conjecture for $\mathfrak{sl}_2$ and two-row partitions has already proven in \cite{RT19,Rho19,IZ21}.

This article aims to provide another  web basis of $\operatorname{Hom}_{\mathfrak{sl}_3(\mathbb{C})}(V^{\otimes 3n},\mathbb{C})$.  These elements are so-called M-diagrams first proposed by Tymoczko \cite{Ty12},  which by definition satisfy the following two properties:
\begin{enumerate}
\item The diagram consists of intersecting M-shaped ``forks''.
\item The left arcs of these M-shaped components do not cross each other.  Neither do the right arcs.
\end{enumerate}
A handful of examples can be found in Example~\ref{mainexample}.  This non-crossing condition can be seen as an attempt to imitate properties of the Templerly-Lieb diagrams.  In  \cite{Ty12},  they were used as an intermediate agent to derive a non-elliptic web from a standard Young tableaux.  In this article,  we treat these diagrams as the main object of study,  and argue they also form a basis (see Theorem~\ref{unitriangular},  Corollary~\ref{alsouni}.)  

\begin{theorem}
M-diagrams with $3n$ boundary points form a basis of $S^{(n,n,n)}$.  Moreover,  there is a unitriangular base change between any pair of the three bases:  the polytabloid basis,  the non-elliptic reduced web basis,  and the M-diagram basis.
\end{theorem}

Web bases of hom-spaces are well-studied.  For example,  cellular bases  \cite{Eli15} and canonical basis  \cite{MPT14, Tu14a} were explicitly constructed.   However,   little is known regarding connections to Kazhdan-Lusztig (KL) theory.   For example,  the standard basis,  KL basis and Deodhar's basis \cite{Deo90} of the Hecke algebra are categorified by the standard bimodules,  the indecomposible Soergel bimodules \cite{Soe90},  and the Bott-Samelson bimodules,  respectively.   Their counterparts via webs still need to be studied.   One link has appeared,  where the work of Tymoczko \cite{Ty12} implies that each non-elliptic web appears as the top summand of the an M-diagram.  We therefore raise the question of whether the transitional matrix is related to the KL polynomials,   in the context of parabolic Hecke modules.   We also hope that the M-diagram basis serves as an intermediate agent in computing the transitioning matrix between the polytabloid basis and the non-elliptic web basis,  and provides machinery in proving the positivity conjecture raised in \cite{RT22}.   

\section*{Acknowledgement}
The author thanks Jonathan Kujawa,  Jaeseong Oh,  Daniel Tubbenhauer and Tianyuan Xu  for helpful discussions.  The author acknowledges the Fonds Spécial de Recherche (FSR) postdoc fellowship at the Université catholique de Louvain for supporting her work,  and IRMP at UC Louvain for the work environment.

\section{Main Results}
We use the notation $(n^3)$ to denote the partition $(n,n,n)$.  A \emph{Young diagram} associated to $(n^3)$ is the collection of $3n$ boxes arranged into a $3\times n$ array.  A \emph{standard Young tableaux} of shape $(n^3)$ is a filling of such a diagram with $1,2,\dots, 3n$ such that the entries increase along each row and column.  Here is one example,  with $n_1=3n-2$,  $n_2=3n-1$,  $n_3=3n$ to save space.
\begin{align}
\ytableausetup{smalltableaux}
T_0= \begin{ytableau}
1 & 4 &\none& \none[\cdots] & \none&n_1\\
2 & 5 &\none& \none[\cdots] & \none&n_2\\
3 & 6 & \none&\none[\cdots] &\none& n_3
\end{ytableau} \label{eq:T0}
\end{align}
Let $\mathcal{T}^n$ be the set of standard Young tableaux of shape $(n^3)$.  The Specht module $S^{(n^3)}$ is the irreducible module for the symmetric group $S_{3n}$ associated to the cycle type $(n^3)$.  It has classic construction with a \emph{polytabloid basis} $\{v_T\}$,  where $T$ ranges over elements of $\mathcal{T}^n$.  We omit the explicit formulation of the action of $S_{3n}$ and refer the reader to \cite{Fu97}.

The group $S_{3n}$ acts on $\mathcal{T}^n$ (on the right) by permuting its entries.  Any element $T\in \mathcal{T}^n$ can be expressed uniquely as $T=T_0.\sigma$ for some $\sigma\in S_{3n}$.  We impose the partial order on $\mathcal{T}^n$ by the left weak Bruhat order on the corresponding permutation $\sigma\in S_{3n}$.  To be more precise,  for $T_1,T_2 \in \mathcal{T}^n$,  $T_1=T_0.\sigma_1$ and $T_2=T_0.\sigma_2$,  we say $T_1\leq T_2$,  if and only if,  for some reduced expressions  of $\sigma_1$ and $\sigma_2$, the former is a subword of a letter with consecutive letters,  starting with the first letter in $\sigma_2$.

The \emph{$\mathfrak{sl}_3$-webs} are planar graphs with directed vertices and directed edges between them.  Each boundary point is incident to exactly one edge;  an internal vertex is always trivalent,  where the orientation on its adjacent edges is in one of the two forms:
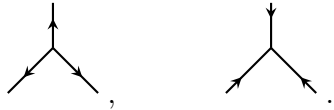
\begin{align*}
\begin{tikzpicture}[scale=0.6]
		\node [] (0) at (0, 0) {};
		\node [] (1) at (0, 1) {};
		\node [] (2) at (-1, -1) {};
		\node [] (3) at (1, -1) {};
		\draw [thick,directed] (0.center) to (2.center);
		\draw [thick,directed] (0.center) to (3.center);
		\draw [thick,directed] (0.center) to (1.center);
\end{tikzpicture},\hspace{.5 in}
\begin{tikzpicture}[scale=0.6]
		\node [] (0) at (0, 0) {};
		\node [] (1) at (0, 1) {};
		\node [] (2) at (-1, -1) {};
		\node [] (3) at (1, -1) {};
		\draw [thick,reverse directed] (0.center) to (2.center);
		\draw [thick,reverse directed] (0.center) to (3.center);
		\draw [thick,reverse directed] (0.center) to (1.center);
\end{tikzpicture}.
\end{align*}
Such diagrams are equivalent up to isotopy and subject to further local relations below:
\begin{align}
\begin{tikzpicture}[baseline={(0,-.2)}, scale = 0.4]
\node at (.5,-1.3) {};
\node at (2.5,-1.3) {};
\draw[thick,  directed] (.5,0) .. controls +(1,+1.5)  .. +(2,0);
\node at (.5,.3) { };
\node at (2.5,.3) { };
\draw[thick, reverse directed] (.5,0) .. controls +(1,-1.5)  .. +(2,0);
\end{tikzpicture} &= \begin{tikzpicture}[baseline={(0,-.2)}, scale = 0.4]
\node at (.5,-1.3) {};
\node at (2.5,-1.3) {};
\draw[thick,  reverse directed] (.5,0) .. controls +(1,+1.5)  .. +(2,0);
\node at (.5,.3) { };
\node at (2.5,.3) { };
\draw[thick, directed] (.5,0) .. controls +(1,-1.5)  .. +(2,0);
\end{tikzpicture} = 3,   \tag{R1}  \\ 
\begin{tikzpicture}[baseline={(0,-.2)}, scale = 0.5]
\node at (-1,0) {};
\draw[thick, directed] (-.5,0) -- (.5,0) node[right] {};
\node at (.5,-1.3) {};
\node at (2.5,-1.3) {};
\draw[thick, reverse directed] (.5,0) .. controls +(1,+1.5)  .. +(2,0);
\node at (.5,.3) { };
\node at (2.5,.3) { };
\draw[thick, reverse directed] (.5,0) .. controls +(1,-1.5)  .. +(2,0);
\draw[thick, directed] (2.5,0) -- (3.5,0) node[right] {};
\end{tikzpicture}
\hspace{1mm}
&= -2  \hspace{2mm}
\begin{tikzpicture}[baseline={(0,-.2)}, scale = 0.5]
\node at (-1,0) {};
\draw[thick, directed] (-.5,0) -- (1.5,0) node[right] {};
\end{tikzpicture},   \tag{R2}   \\ 
\begin{tikzpicture}[baseline={(0,-.2)}, scale = 0.5]
\node at (.5,-1.3) {};
\node at (2.5,-1.3) {};
\draw[thick, directed] (.5,.5) .. controls +(1,+.5)  .. +(2,0);
\node at (.5,.3) { };
\node at (2.5,.3) { };
\draw[thick, reverse directed] (.5,-.5) .. controls +(1,-.5)  .. +(2,0);
\node at (2.5,.3) {};
\draw[thick, directed] (2.5,-.5) -- (2.5,.5) node[right] {};
\node at (-1,0) {};
\draw[thick, reverse directed] (.5,-.5) -- (.5,.5) node[right] {};
\node at (-.3,1.3) {};
\draw[thick,reverse directed] (0,1) -- (.5,.5) node[right] {};
\node at (3.3,1.3) {};
\draw[thick, reverse directed] (2.5,.5) -- (3,1) node[right] {};
\node at (-.3,-1.3) {};
\draw[thick, reverse directed] (.5,-.5) -- (0,-1) node[right] {};
\node at (3.3,-1.3) {};
\draw[thick, directed] (2.5,-.5) -- (3,-1) node[right] {};
\end{tikzpicture}
&= 
\begin{tikzpicture}[baseline={(0,-.2)}, scale = 0.5]
\node at (2.5,.8) {};
\node at (2.5,-.8) {};
\draw[thick, reverse directed] (2.5,-.5) -- (2.5,.5) node[right] {};
\node at (1.5,-.8) {};
\node at (1.5,.8) {};
\draw[thick, directed] (1.5,-.5) -- (1.5,.5) node[right] {};
\end{tikzpicture}
\hspace{2mm}
+ 
\hspace{2mm}
\begin{tikzpicture}[baseline={(0,-.2)}, scale = 0.6]
\node at (2.5,.8) {};
\node at (2.5,-.8) {};
\draw[thick,  reverse directed] (2.5,-.5) -- (1.5,-.5) node[right] {};
\node at (1.5,-.8) {};
\node at (1.5,.8) {};
\draw[thick, reverse directed] (1.5,.5) -- (2.5,.5) node[right] {};
\end{tikzpicture}.\tag{R3}
\end{align}

For this article we will only consider webs whose boundary points all lie in a horizontal line,  drawn at the bottom of the picture.  We further require that the orientation of the boundary points flow out of the web and into this horizontal line.   We will assume this orientation throughout the article and omit it later on.  Let $D^{(n^3)}$ be the $\mathbb{C}$-vector space spanned by these webs.   A web is called \emph{non-elliptic},  or \emph{reduced},  if none of the local relations (R1,R2,R3) can be applied.  Let $\mathcal{W}^n$ be the set of reduced webs with $3n$ boundary points at the bottom.  It has been shown in \cite{PPR09} that $\mathcal{W}^n$ form a basis of  $D^{(n^3)}$.  An example of such a web is 
\begin{align*}
W_0=\begin{tikzpicture}[baseline=0]
		\node [] (3) at (-1, -0.5) {};
		\node [] (4) at (0, 0.5) {};
		\node [] (5) at (1, -0.5) {};
		\node [] (6) at (0, -0.5) {};
		\draw [thick, directed,  in=90,  out=180] (4.center) to (3.center);
		\draw [thick, directed,  in=90,  out=0](4.center) to (5.center);
		\draw  [thick, directed] (4.center) to (6.center);
\end{tikzpicture}\:
\begin{tikzpicture}[baseline=0,scale=0.8]
		\node [] (3) at (-1, -0.5) {};
		\node [] (4) at (0, 0.5) {};
		\node [] (5) at (1, -0.5) {};
		\node [] (6) at (0, -0.5) {};
		\draw [thick, directed,  in=90,  out=180] (4.center) to (3.center);
		\draw [thick, directed,  in=90,  out=0](4.center) to (5.center);
		\draw  [thick, directed] (4.center) to (6.center);
\end{tikzpicture}
\cdots
\begin{tikzpicture}[baseline=0,scale=0.8]
		\node [] (3) at (-1, -0.5) {};
		\node [] (4) at (0, 0.5) {};
		\node [] (5) at (1, -0.5) {};
		\node [] (6) at (0, -0.5) {};
		\draw [thick, directed,  in=90,  out=180] (4.center) to (3.center);
		\draw [thick, directed,  in=90,  out=0](4.center) to (5.center);
		\draw  [thick, directed] (4.center) to (6.center);
\end{tikzpicture}.
\end{align*}

The space $D^{(n^3)}$ admits a right action of the symmetric group,  where each simple transiposition $s_i=(i \: j)$,  $j=i+1$,  acts via stacking a swap at the bottom of the picture,  as illustrated below,  where the orientation flows from the diagram $D$.  The crossing should be interpreted via the skein relation below.    
\begin{align*}
\begin{tikzpicture}[baseline={(0,-.5)}, scale = 0.8]
\draw[dotted] (-.25,0) -- (3,0) -- (3,-.9) -- (-.25,-.9) -- cycle;
\node at (1.3,-.4) {$D$};
\end{tikzpicture}\hspace{.1 in}.s_i
=
\begin{tikzpicture}[baseline={(0,-.3)}, yscale = -0.8]
\draw[dotted] (-.25,0) -- (3,0) -- (3,-.9) -- (-.25,-.9) -- cycle;
\node at (1.3,-.4) {$D$};
\draw[thick] (0,0) .. controls +(0,.5)  .. +(0,1);
\draw[thick] (0.9,0) .. controls +(0,.5)  .. +(0,1);
\begin{footnotesize}
\node at (.5,.5) {$\cdots$};
\node at (2.5,.5) {$\cdots$};
\node at (1.2,1.3) {$i$};
\node at (1.7,1.3) {$j$};
\draw[thick] (2,0) .. controls +(0,.5)  .. +(0,1);
\draw[thick] (2.8,0) .. controls +(0,.5)  .. +(0,1);
\end{footnotesize}
\draw[thick] (1.2,0) .. controls +(.25,.5)  .. +(.5,1);
\draw[thick] (1.7,0) .. controls +(-.25,.5)  .. +(-.5,1);
\end{tikzpicture},\hspace{.5 in}
\begin{tikzpicture}[baseline=0,scale=0.4]
		\node  (0) at (-1, 1.5) {};
		\node  (1) at (1, -1.5) {};
		\node (2) at (-1, -1.5) {};
		\node  (3) at (1, 1.5) {};
		\draw  [thick,  reverse directed](0.center) to (1.center);
		\draw  [thick, reverse directed](3.center) to (2.center);
\end{tikzpicture}=
\begin{tikzpicture}[baseline=0,scale=0.4]
		\node (0) at (-1, 1.5) {};
		\node (1) at (0, 0.5) {};
		\node (2) at (1, 1.5) {};
		\node (3) at (-1, -1.5) {};
		\node (4) at (1, -1.5) {};
		\node (5) at (0, -0.5) {};
		\draw [thick, directed] (1.center) to (0.center);
		\draw [thick, directed](1.center) to (2.center);
		\draw [thick, directed] (1.center) to (5.center);
		\draw [thick, directed](3.center) to (5.center);
		\draw [thick, directed](4.center) to (5.center);
\end{tikzpicture}
 + \begin{tikzpicture}[baseline=0,scale=0.4]
		\node  (0) at (-1, -1.5) {};
		\node  (1) at (-1, 1.5) {};
		\node  (2) at (1, -1.5) {};
		\node  (3) at (1, 1.5) {};
		\draw [thick, directed] (0.center) to (1.center);
		\draw [thick, directed] (2.center) to (3.center);
\end{tikzpicture}.
\end{align*}

We first cite some known results.
\begin{theorem}{\cite[Lemma~4.2]{PPR09},\cite[Theorem~3.9]{IZ19}, \cite[Theorem~6.1]{RT22}}
There is an isomorphism  $S^{(n^3)}\simeq D^{(n^3)}$ as right $S_{3n}$-modules,  sending $v_{T_0}$ to $W_0$. 
\end{theorem}
\begin{theorem}{\cite[Theorem~6.4]{RT22}}\label{unitriangular}
The base change from $\{v_T\}_{T\in \mathcal{T}^n}$ to $\mathcal{W}^n$ is unitriangular.  
\end{theorem}

Let $S_{(3^n)}$ be the group of columm stablizers (i.e.  permutations which stabilizes the set of entries in each column) of $T_0$ in (\ref{eq:T0}).  Similarly,  let $S_{(n^3)}$ be the group of row stablizers of $T_0$.  They are  both conjugate to the Young subgroups  associated to the partition $(3^n)$,  $(n^3)$,  respectively.  Let $\mathcal{D}^{(3^n)}$ and $\mathcal{D}^{(n^3)}$  be the sets of minimal-length right coset representatives for $S_{(3^n)}$ and $S_{(n^3)}$,  respectively.

A \emph{fork diagram} is a web of the form $W_0.\sigma$ for 
 $\sigma\in \mathcal{D}^{(3^n)}$.   This is to say that $\sigma$ does not permute the relative order of the three boundary points within each arc in $W_0$.  Pictorially,  a fork diagram is such that any of its internal vertices is adjacent to  three  boundary points.  One can refer to  Examples~\ref{mainexample}  for a pictorial illustration.  
 
We label the endpoints of a web by $1,2,\dots, 3n$,  from left to right.    Given a fork diagram $w$ and positive integers $a<b<c$, we call $(a,b,c)$ an \emph{arc} in $w$ if the boundary vertices $a,b,c$ are connected by a trivalent vertex in $w$.   It is straightforward to see that any boundary point of a fork diagram belongs to some arc.  We will also refer to $(a,b)$ as a \emph{left} arc and $(b,c)$ a \emph{right} arc.  We will refer to $a,b,c$ as the \emph{left,  middle},  and \emph{right endpoints} of the arc,  respectively. 
 
 A fork diagram $W_0.\sigma$ is further called a \emph{polytabloid diagram} if $\sigma\in \mathcal{D}^{(n^3)}$.   In other words,  $\sigma$ also does not permute the relative order of all left  (middle,  right,  respectively) endpoints.  We denote the set of polytabloid diagram with $3n$ boundary points by $\mathcal{P}^n$.  It is  a straightforward combinatorial fact that $\mathcal{D}^{(3^n)}\cap\mathcal{D}^{(n^3)}$ is in bijection with $\mathcal{T}^n$,  therefore we have a bijection of sets $\mathcal{T}^n\simeq \mathcal{P}^n $.  This bijection will be explained explicitly,  after we set up some notations.

The following lemmas show that fork diagrams are well-defined up to their arcs, i.e.   the partition of $\{1,2,\dots,3n\}$ into $3$-element subsets.   These results are well-known; for example one could verify them directly using the skein relation.

\begin{lemma} The following relations hold.
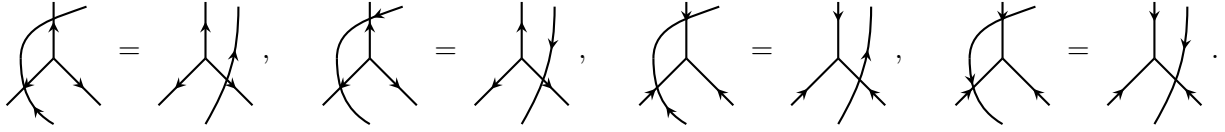
\begin{equation*}
\begin{tikzpicture}[scale=0.25,baseline=0]
		\node [] (0) at (0, 0) {};
		\node [] (1) at (0, 3) {};
		\node [] (2) at (2.5, -2.5) {};
		\node [] (3) at (-2.5, -2.5) {};
		\node [] (4) at (1.75, 2.75) {};
		\node [] (5) at (-1.75, 0) {};
		\node [] (6) at (0, -3.5) {};
		\draw[thick,directed] (0.center) to (1.center);
		\draw[thick,directed] (0.center) to (3.center);
		\draw[thick,directed] (0.center) to (2.center);
		\draw[thick,out=90,in=200] (5.center) to (4.center);
		\draw[thick,in=150,out=270,  reverse directed] (5.center) to (6.center);
\end{tikzpicture}
=
\begin{tikzpicture}[scale=0.25,baseline=0]
		\node [] (0) at (0, 0) {};
		\node [] (1) at (0, 3) {};
		\node [] (2) at (2.5, -2.5) {};
		\node [] (3) at (-2.5, -2.5) {};
		\node [] (4) at (1.75, 2.75) {};
		\node [] (5) at (-1.75, 0) {};
		\node [] (6) at (0, -3.5) {};
		\draw[thick,directed] (0.center) to (1.center);
		\draw[thick,directed] (0.center) to (3.center);
		\draw[thick,directed] (0.center) to (2.center);
		\draw[thick,out=60,in=270, directed] (6.center) to (4.center);
\end{tikzpicture},
\hspace{.2 in}
\begin{tikzpicture}[scale=0.25,baseline=0]
		\node [] (0) at (0, 0) {};
		\node [] (1) at (0, 3) {};
		\node [] (2) at (2.5, -2.5) {};
		\node [] (3) at (-2.5, -2.5) {};
		\node [] (4) at (1.75, 2.75) {};
		\node [] (5) at (-1.75, 0) {};
		\node [] (6) at (0, -3.5) {};
		\draw[thick,directed] (0.center) to (1.center);
		\draw[thick,directed] (0.center) to (3.center);
		\draw[thick,directed] (0.center) to (2.center);
		\draw[thick,out=90,in=200,reverse directed] (5.center) to (4.center);
		\draw[thick,out=150,in=270] (6.center) to (5.center);
\end{tikzpicture}
=
\begin{tikzpicture}[scale=0.25,baseline=0]
		\node [] (0) at (0, 0) {};
		\node [] (1) at (0, 3) {};
		\node [] (2) at (2.5, -2.5) {};
		\node [] (3) at (-2.5, -2.5) {};
		\node [] (4) at (1.75, 2.75) {};
		\node [] (5) at (-1.75, 0) {};
		\node [] (6) at (0, -3.5) {};
		\draw[thick,directed] (0.center) to (1.center);
		\draw[thick,directed] (0.center) to (3.center);
		\draw[thick,directed] (0.center) to (2.center);
		\draw[thick,out=60,in=270, reverse directed] (6.center) to (4.center);
\end{tikzpicture}, \hspace{.2 in}
\begin{tikzpicture}[scale=0.25,baseline=0]
		\node [] (0) at (0, 0) {};
		\node [] (1) at (0, 3) {};
		\node [] (2) at (2.5, -2.5) {};
		\node [] (3) at (-2.5, -2.5) {};
		\node [] (4) at (1.75, 2.75) {};
		\node [] (5) at (-1.75, 0) {};
		\node [] (6) at (0, -3.5) {};
		\draw[thick,reverse directed] (0.center) to (1.center);
		\draw[thick,reverse directed] (0.center) to (3.center);
		\draw[thick,reverse directed] (0.center) to (2.center);
		\draw[thick,out=90,in=200] (5.center) to (4.center);
		\draw[thick,in=150,out=270,reverse directed] (5.center) to (6.center);
\end{tikzpicture}
=
\begin{tikzpicture}[scale=0.25,baseline=0]
		\node [] (0) at (0, 0) {};
		\node [] (1) at (0, 3) {};
		\node [] (2) at (2.5, -2.5) {};
		\node [] (3) at (-2.5, -2.5) {};
		\node [] (4) at (1.75, 2.75) {};
		\node [] (5) at (-1.75, 0) {};
		\node [] (6) at (0, -3.5) {};
		\draw[thick,reverse directed] (0.center) to (1.center);
		\draw[thick,reverse directed] (0.center) to (3.center);
		\draw[thick,reverse directed] (0.center) to (2.center);
		\draw[thick,out=60,in=270, directed] (6.center) to (4.center);
\end{tikzpicture},
\hspace{.2 in}
\begin{tikzpicture}[scale=0.25,baseline=0]
		\node [] (0) at (0, 0) {};
		\node [] (1) at (0, 3) {};
		\node [] (2) at (2.5, -2.5) {};
		\node [] (3) at (-2.5, -2.5) {};
		\node [] (4) at (1.75, 2.75) {};
		\node [] (5) at (-1.75, 0) {};
		\node [] (6) at (0, -3.5) {};
		\draw[thick,reverse directed] (0.center) to (1.center);
		\draw[thick,reverse directed] (0.center) to (3.center);
		\draw[thick,reverse directed] (0.center) to (2.center);
		\draw[thick,out=90,in=200] (5.center) to (4.center);
		\draw[thick,out=150,in=270,reverse directed] (6.center) to (5.center);
\end{tikzpicture}
=
\begin{tikzpicture}[scale=0.25,baseline=0]
		\node [] (0) at (0, 0) {};
		\node [] (1) at (0, 3) {};
		\node [] (2) at (2.5, -2.5) {};
		\node [] (3) at (-2.5, -2.5) {};
		\node [] (4) at (1.75, 2.75) {};
		\node [] (5) at (-1.75, 0) {};
		\node [] (6) at (0, -3.5) {};
		\draw[thick,reverse directed] (0.center) to (1.center);
		\draw[thick,reverse directed] (0.center) to (3.center);
		\draw[thick,reverse directed] (0.center) to (2.center);
		\draw[thick,out=60,in=270, reverse directed] (6.center) to (4.center);
\end{tikzpicture}.
\end{equation*}
\end{lemma}
\begin{lemma}\label{lem:R12}
The following relations,  famously known as the Reidemeister I and Reidemeister II relations,  also hold.
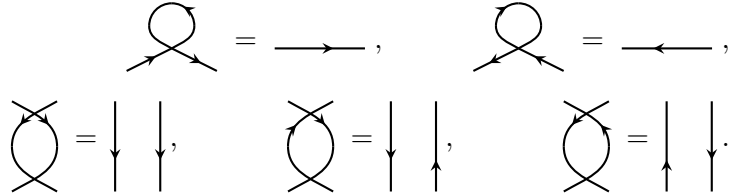
\begin{align*}
\begin{tikzpicture}[scale=0.3,baseline=0]
		\node [] (0) at (-2, -1) {};
		\node [] (1) at (2, -1) {};
		\node [] (2) at (0, 2) {};
		\node [] (3) at (0, 0) {};
		\node [] (4) at (1, 1) {};
		\node [] (5) at (-1, 1) {};
		\draw [thick,directed] (0.center) to (3.center);
		\draw [thick,out=30,  in=270] (3.center) to (4.center);
		\draw [thick,directed,  in=0, out=90] (4.center) to (2.center);
		\draw [thick, in=90, out=180] (2.center) to (5.center);
		\draw [thick, in=150, out=270] (5.center) to (3.center);
		\draw [thick,directed] (3.center) to (1.center);
\end{tikzpicture}
=
\begin{tikzpicture}[scale=0.3,baseline=0]
		\node [] (0) at (-2, 0) {};
		\node [] (1) at (2, 0) {};
		\draw [thick,directed] (0.center) to (1.center);
\end{tikzpicture},   \hspace{.4 in}
\begin{tikzpicture}[xscale=-0.3,  yscale=0.3,  baseline=0]
		\node [] (0) at (-2, -1) {};
		\node [] (1) at (2, -1) {};
		\node [] (2) at (0, 2) {};
		\node [] (3) at (0, 0) {};
		\node [] (4) at (1, 1) {};
		\node [] (5) at (-1, 1) {};
		\draw [thick,directed] (0.center) to (3.center);
		\draw [thick,out=30,  in=270] (3.center) to (4.center);
		\draw [thick,directed,  in=0, out=90] (4.center) to (2.center);
		\draw [thick, in=90, out=180] (2.center) to (5.center);
		\draw [ thick,in=150, out=270] (5.center) to (3.center);
		\draw [thick,directed] (3.center) to (1.center);
\end{tikzpicture}
=
\begin{tikzpicture}[xscale=-0.3,  yscale=0.3,   baseline=0]
		\node [] (0) at (-2, 0) {};
		\node [] (1) at (2, 0) {};
		\draw [thick,directed] (0.center) to (1.center);
\end{tikzpicture},   \\
\begin{tikzpicture} [scale=0.3,  baseline=0]
		\node [] (0) at (-1, 2) {};
		\node [] (1) at (1, 0) {};
		\node [] (2) at (-1, -2) {};
		\node [] (3) at (1, 2) {};
		\node [] (4) at (-1, 0) {};
		\node [] (5) at (1, -2) {};
		\draw [thick,directed,  out=210,  in=90] (3.center) to (4.center);
		\draw [thick,out=270,  in=150] (4.center) to (5.center);
		\draw [thick,directed,  out=330,  in=90] (0.center) to (1.center);
		\draw [thick,out=270,  in=30](1.center) to (2.center);
\end{tikzpicture}=\begin{tikzpicture} [scale=0.3,  baseline=0]
		\node [] (0) at (-1, 2) {};
		\node [] (1) at (1, 0) {};
		\node [] (2) at (-1, -2) {};
		\node [] (3) at (1, 2) {};
		\node [] (4) at (-1, 0) {};
		\node [] (5) at (1, -2) {};
		\draw [thick,directed] (3.center) to (5.center);
		\draw [thick,directed] (0.center) to (2.center);
\end{tikzpicture} ,\hspace{.5 in}
\begin{tikzpicture} [scale=0.3,  baseline=0]
		\node [] (0) at (-1, 2) {};
		\node [] (1) at (1, 0) {};
		\node [] (2) at (-1, -2) {};
		\node [] (3) at (1, 2) {};
		\node [] (4) at (-1, 0) {};
		\node [] (5) at (1, -2) {};
		\draw [thick,reverse directed,  out=210,  in=90] (3.center) to (4.center);
		\draw [thick,out=270,  in=150] (4.center) to (5.center);
		\draw [thick,directed,  out=330,  in=90] (0.center) to (1.center);
		\draw [thick,out=270,  in=30](1.center) to (2.center);
\end{tikzpicture}=
\begin{tikzpicture} [scale=0.3,  baseline=0]
		\node [] (0) at (-1, 2) {};
		\node [] (1) at (1, 0) {};
		\node [] (2) at (-1, -2) {};
		\node [] (3) at (1, 2) {};
		\node [] (4) at (-1, 0) {};
		\node [] (5) at (1, -2) {};
		\draw [thick,reverse directed] (3.center) to (5.center);
		\draw [thick,directed] (0.center) to (2.center);
\end{tikzpicture},  \hspace{.5 in}
\begin{tikzpicture} [scale=0.3,  baseline=0]
		\node [] (0) at (-1, 2) {};
		\node [] (1) at (1, 0) {};
		\node [] (2) at (-1, -2) {};
		\node [] (3) at (1, 2) {};
		\node [] (4) at (-1, 0) {};
		\node [] (5) at (1, -2) {};
		\draw [thick,directed,  out=210,  in=90] (3.center) to (4.center);
		\draw [thick,out=270,  in=150] (4.center) to (5.center);
		\draw [thick,reverse directed,  out=330,  in=90] (0.center) to (1.center);
		\draw [thick,out=270,  in=30](1.center) to (2.center);
\end{tikzpicture}=\begin{tikzpicture} [scale=0.3,  baseline=0]
		\node [] (0) at (-1, 2) {};
		\node [] (1) at (1, 0) {};
		\node [] (2) at (-1, -2) {};
		\node [] (3) at (1, 2) {};
		\node [] (4) at (-1, 0) {};
		\node [] (5) at (1, -2) {};
		\draw [thick,directed] (3.center) to (5.center);
		\draw [thick,reverse directed] (0.center) to (2.center);
\end{tikzpicture}.
\end{align*}
\end{lemma}
We remark that two arcs in a fork diagrams cross at most once after one applies the Reidemeister II moves.

Given two arcs $a=(a_1,b_1,c_1)$ and $b=(a_2,b_2,c_2)$ in $m$,  we say they satisfy the partial-non-crossing property if both of the following holds.
\begin{itemize}
    \item Either $a_1<a_2<b_2<b_1$, or $a_2<a_1<b_1<b_2$.
    \item Either $b_1<b_2<c_2<c_1$, or $b_2<b_1<c_1<c_2$.
\end{itemize}
In short,  left arcs in $a$ and $b$ do not intersect,  and neither do their right arcs.   A fork diagram $m$ is called an \emph{M-diagram} if any pair of arcs in $m$ satisfy the partial-non-crossing property.  Let $\mathcal{M}^n$ be the set of M-diagrams with $3n$ boundary points.

 We now illustrate the bijection $\phi: \mathcal{T}^n\simeq \mathcal{P}^n $ mentioned earlier.   Let $T\in \mathcal{T}^n$,  and $\phi(T)$ be the fork diagram,  so that whenever two  (labeled) endpoints lie in the same column in $T$,  they are connected to the same internal vertex in $\phi(T)$. 
\begin{lemma}
$\phi(T)$ is a polytabloid diagram.
\end{lemma} 
 \begin{proof}
This is because the relative order is preserved within  the three (left,  middle,  right) endpoints of each arc,  because of the standard condition on each column  in $T$.  Similarly,  the relative order is preserved within all left endpoints,  because of the standard condition on each row in $T$. 
\end{proof}
It follows that $\phi(T)=W_0.\sigma$ with $\sigma \in D_{(3^n)}\cap D_{(n^3)}$,  and $\phi$ is a bijection.  We transfer the previously defined partial order from $\mathcal{T}^n$ to  $\mathcal{P}^n$ so that $\phi$ is a bijection of partially ordered set.   A three-row analogue of \cite[Proposition~2.5]{IZ19} works verbatim  and leads to the following result.

\begin{theorem}
The map  $v_T \mapsto \phi(T)$ defines a module isomorphism $\mathcal{D}^{(n^3)}\simeq \mathcal{W}^{(n^3)}$.
\end{theorem}

We also recall a bijection between $\mathcal{T}^n$ and $\mathcal{M}^n$ defined in \cite{Ty12}.  Let $T\in \mathcal{T}^n$.  We first label the points on the horizontal line by integers $1,2,\dots, 3n$.  Next,  if $i$ is in the first row of $T$,  we mark it with $+$;  if it is in the middle row of $T$,  we mark it with $0$; if it is in the last row of $T$ we mark it with $-$.  Now we look at all vertices marked with $+$ and $0$ and draw non-intersecting arcs between them,  so that any point marked with $+$ corresponds to a left endpoint,  and any point marked with $0$ corresponds to a middle endpoint.  It is well-known that such a diagram is unique.  Next,  we ignore the arcs currently drawn and look at all endpoints marked with $0$ and $-$.  We now draw non-intersecting arcs between them,  so that $0$'s correspond to a middle endpoints and $-$'s correspond to a right endpoints.  We then combine the arcs drawn at the two stages,  and modify the boundary points labelled with $0$ by the following procedure to obtain $\psi(T)\in\mathcal{M}^n$:
\begin{align*}
\begin{tikzpicture}[baseline=0,  scale=.5]
		\node  (0) at (0, -1) {};
		\node  (1) at (-1, 1) {};
		\node  (2) at (1, 1) {};
		\draw [thick,  in=0,  out = 90] (0.center) to (1.center);
		\draw [thick,  in=180,  out = 90]  (0.center) to (2.center);
\end{tikzpicture} \longrightarrow
\begin{tikzpicture}[baseline=0,  scale=.5]
		\node  (0) at (-1, 1) {};
		\node  (1) at (0, 0) {};
		\node  (2) at (0, -1) {};
		\node  (3) at (1, 1) {};
		\draw [thick,  directed] (1.center) to (2.center);
		\draw [thick,  directed](1.center) to (0.center);
		\draw [thick,  directed](1.center) to (3.center);
\end{tikzpicture}.
\end{align*}
\begin{remark}
Compared to the M-diagrams in \cite{Ty12},  we leave the crossings in the M-diagrams intact (rather than replacing them by the top summand in the skein relation).
\end{remark}

We transfer the partial order from $\mathcal{T}^n$ to  $\mathcal{M}^n$ in a way that $\psi$ is a bijection of partially ordered set.   We denote this partial order on $\mathcal{M}^n$ by $\leq$.  This order is called the shadow containment order and is extensively studied in detail in \cite{RT22}.  However,  we will not be using the combinatorial description used on $\mathcal{M}^n$,  as introduced in  \cite{RT22},   and will only refer to the definition via $\mathcal{T}^n$.


\ytableausetup{smalltableaux}
\begin{example}\label{mainexample}
For $n=2$,  we list all standard Young tableaux of shape $(2^3)$,   as well as  the corresponding polytabloid diagrams and M-diagrams,  with $\phi(T_i)=v_i$ and $\psi(T_i)=m_i$.
\begin{center}
\begin{tabular}{|c|c|c|c|c|}
\hline
    $T_0$ & $T_1$ &$T_2$&$T_3$&$T_4$ \\
   \begin{ytableau}
   1 & 4 \\ 2 & 5 \\ 3 & 6
   \end{ytableau} &
   \begin{ytableau}
   1 & 3 \\ 2 & 5 \\ 4 & 6
   \end{ytableau}
   &
    \begin{ytableau}
   1 & 2 \\ 3 & 5 \\ 4 & 6
   \end{ytableau}
   &
    \begin{ytableau}
   1 & 3 \\ 2 & 4 \\ 5 & 6
   \end{ytableau}
   &
   \begin{ytableau}
   1 & 2 \\ 3 & 4 \\ 5 & 6
   \end{ytableau} \\
   &&&&\\
   \hline
    $v_0$ & $v_1$ &$v_2$&$v_3$&$v_4$ \\
    \begin{tikzpicture}[scale=0.2]
		\node [] (0) at (-1, -1) {};
		\node [] (1) at (-3, 0) {};
		\node [] (2) at (-5, -1) {};
		\node [] (3) at (1, -1) {};
		\node [] (4) at (3, 0) {};
		\node [] (5) at (5, -1) {};
		\node [] (6) at (-4, 1) {};
		\node [] (7) at (-2, 1) {};
		\node [] (8) at (2, 1) {};
		\node [] (9) at (4, 1) {};
		\node [] (10) at (-3, -1) {};
		\node [] (11) at (3, -1) {};
		\draw[out=180, in=80, directed] (6.center) to (2.center);
		\draw[out=0, in=120] (6.center) to (1.center);
		\draw[out=60, in=180] (1.center) to (7.center);
		\draw[out=0, in=100, directed] (7.center) to (0.center);
		\draw[directed] (1.center) to (10.center);
		\draw[out=180, in=80, directed] (8.center) to (3.center);
		\draw[out=0, in=120] (8.center) to (4.center);
		\draw[out=60, in=180] (4.center) to (9.center);
		\draw[out=0, in=100, directed] (9.center) to (5.center);
		\draw[directed] (4.center) to (11.center);
\end{tikzpicture}
    &
     \begin{tikzpicture}[scale=0.2]
		\node [] (0) at (-1, -1) {};
		\node [] (1) at (-3, 0) {};
		\node [] (2) at (-5, -1) {};
		\node [] (3) at (1, -1) {};
		\node [] (4) at (3, 0) {};
		\node [] (5) at (5, -1) {};
		\node [] (6) at (-4, 1) {};
		\node [] (7) at (-2, 1) {};
		\node [] (8) at (2, 1) {};
		\node [] (9) at (4, 1) {};
		\node [] (10) at (-3, -1) {};
		\node [] (11) at (3, -1) {};
		\draw[out=180, in=80, directed] (6.center) to (2.center);
		\draw[out=0, in=120] (6.center) to (1.center);
		\draw[out=60, in=180] (1.center) to (7.center);
		\draw[out=0, in=120, directed] (7.center) to (3.center);
		\draw[directed] (1.center) to (10.center);
		\draw[in=180, out=60, reverse directed] (0.center) to (8.center);
		\draw[out=0, in=120] (8.center) to (4.center);
		\draw[out=60, in=180] (4.center) to (9.center);
		\draw[out=0, in=100, directed] (9.center) to (5.center);
		\draw[directed] (4.center) to (11.center);
\end{tikzpicture}
    &
    \begin{tikzpicture}[scale=0.2]
		\node [] (0) at (-1, -1) {};
		\node [] (1) at (-3, -1) {};
		\node [] (2) at (-5, -1) {};
		\node [] (3) at (1, -1) {};
		\node [] (4) at (3, -1) {};
		\node [] (5) at (5, -1) {};
		\node [] (6) at (0, 0.2) {};
		\node [] (7) at (4, 0.2) {};
		\node [] (8) at (-3, 1) {};
		\node [] (9) at (1, 1) {};
		\draw[directed] (6.center) to (0.center);
		\draw[directed] (6.center) to (3.center);
		\draw[directed] (7.center) to (4.center);
		\draw[directed] (7.center) to (5.center);
		\draw[directed, out=180, in=80] (8.center) to (2.center);
		\draw[out=0, in=120] (8.center) to (6.center);
		\draw[directed, out=180, in=80] (9.center) to (1.center);
		\draw[out=0, in=120] (9.center) to (7.center);
\end{tikzpicture}
    &
    \begin{tikzpicture}[yscale=0.2, xscale=-0.2]
		\node [] (0) at (-1, -1) {};
		\node [] (1) at (-3, -1) {};
		\node [] (2) at (-5, -1) {};
		\node [] (3) at (1, -1) {};
		\node [] (4) at (3, -1) {};
		\node [] (5) at (5, -1) {};
		\node [] (6) at (0, 0.2) {};
		\node [] (7) at (4, 0.2) {};
		\node [] (8) at (-3, 1) {};
		\node [] (9) at (1, 1) {};
		\draw[directed] (6.center) to (0.center);
		\draw[directed] (6.center) to (3.center);
		\draw[directed] (7.center) to (4.center);
		\draw[directed] (7.center) to (5.center);
		\draw[directed, out=180, in=80] (8.center) to (2.center);
		\draw[out=0, in=120] (8.center) to (6.center);
		\draw[directed, out=180, in=80] (9.center) to (1.center);
		\draw[out=0, in=120] (9.center) to (7.center);
\end{tikzpicture}
    &
    \begin{tikzpicture}[scale=0.2]
		\node [] (0) at (-1, -1) {};
		\node [] (1) at (-3, -1) {};
		\node [] (2) at (-5, -1) {};
		\node [] (3) at (1, -1) {};
		\node [] (4) at (3, -1) {};
		\node [] (5) at (5, -1) {};
		\node [] (6) at (2, 0) {};
		\node [] (7) at (-3, 1) {};
		\node [] (8) at (-2, 0) {};
		\node [] (9) at (3, 1) {};
		\draw[directed] (6.center) to (0.center);
		\draw[directed] (6.center) to (4.center);
		\draw[out=180, in=80, directed] (7.center) to (2.center);
		\draw[out=0, in=130] (7.center) to (6.center);
		\draw[directed] (8.center) to (3.center);
		\draw[directed] (8.center) to (1.center);
		\draw[out=180,in=50] (9.center) to (8.center);
		\draw[directed, out=0, in=80] (9.center) to (5.center);
\end{tikzpicture}
    \\
    \hline
    $m_0$ & $m_1$ & $m_2$ & $m_3$ & $m_4$ \\
    \begin{tikzpicture}[scale=0.2]
		\node [] (0) at (-1, -1) {};
		\node [] (1) at (-3, 0) {};
		\node [] (2) at (-5, -1) {};
		\node [] (3) at (1, -1) {};
		\node [] (4) at (3, 0) {};
		\node [] (5) at (5, -1) {};
		\node [] (6) at (-4, 1) {};
		\node [] (7) at (-2, 1) {};
		\node [] (8) at (2, 1) {};
		\node [] (9) at (4, 1) {};
		\node [] (10) at (-3, -1) {};
		\node [] (11) at (3, -1) {};
		\draw[out=180, in=80, directed] (6.center) to (2.center);
		\draw[out=0, in=120] (6.center) to (1.center);
		\draw[out=60, in=180] (1.center) to (7.center);
		\draw[out=0, in=100, directed] (7.center) to (0.center);
		\draw[directed] (1.center) to (10.center);
		\draw[out=180, in=80, directed] (8.center) to (3.center);
		\draw[out=0, in=120] (8.center) to (4.center);
		\draw[out=60, in=180] (4.center) to (9.center);
		\draw[out=0, in=100, directed] (9.center) to (5.center);
		\draw[directed] (4.center) to (11.center);
\end{tikzpicture}
    &
    \begin{tikzpicture}[scale=0.2]
		\node [] (0) at (-1, -1) {};
		\node [] (1) at (-3, 0) {};
		\node [] (2) at (-5, -1) {};
		\node [] (3) at (1, -1) {};
		\node [] (4) at (3, 0) {};
		\node [] (5) at (5, -1) {};
		\node [] (6) at (-4, 1) {};
		\node [] (7) at (-2, 1) {};
		\node [] (8) at (2, 1) {};
		\node [] (9) at (4, 1) {};
		\node [] (10) at (-3, -1) {};
		\node [] (11) at (3, -1) {};
		\draw[out=180, in=80, directed] (6.center) to (2.center);
		\draw[out=0, in=120] (6.center) to (1.center);
		\draw[out=60, in=180] (1.center) to (7.center);
		\draw[out=0, in=120, directed] (7.center) to (3.center);
		\draw[directed] (1.center) to (10.center);
		\draw[in=180, out=60, reverse directed] (0.center) to (8.center);
		\draw[out=0, in=120] (8.center) to (4.center);
		\draw[out=60, in=180] (4.center) to (9.center);
		\draw[out=0, in=100, directed] (9.center) to (5.center);
		\draw[directed] (4.center) to (11.center);
\end{tikzpicture}
    &
    \begin{tikzpicture}[scale=0.2]
		\node [] (0) at (-1, -1) {};
		\node [] (1) at (-3, -1) {};
		\node [] (2) at (-5, -1) {};
		\node [] (3) at (1, -1) {};
		\node [] (4) at (3, -1) {};
		\node [] (5) at (5, -1) {};
		\node [] (6) at (-1, 0) {};
		\node [] (7) at (2, 1) {};
		\node [] (8) at (-3, 0.5) {};
		\draw[directed, out=0, in=90] (7.center) to (5.center);
		\draw[directed] (7.center) to (4.center);
		\draw[directed] (6.center) to (0.center);
		\draw[directed, out=180, in=90] (6.center) to (1.center);
		\draw[directed, out=0, in=90] (6.center) to (3.center);
		\draw[directed, out=180, in=90] (7.center) to (2.center);
\end{tikzpicture}
    &
    \begin{tikzpicture}[yscale=0.2,xscale=-0.2]
		\node  (0) at (-1, -1) {};
		\node  (1) at (-3, -1) {};
		\node  (2) at (-5, -1) {};
		\node  (3) at (1, -1) {};
		\node  (4) at (3, -1) {};
		\node  (5) at (5, -1) {};
		\node  (6) at (-1, 0) {};
		\node  (7) at (2, 1) {};
		\node  (8) at (-3, 0.5) {};
		\draw[directed, out=0, in=90] (7.center) to (5.center);
		\draw[directed] (7.center) to (4.center);
		\draw[directed] (6.center) to (0.center);
		\draw[directed, out=180, in=90] (6.center) to (1.center);
		\draw[directed, out=0, in=90] (6.center) to (3.center);
		\draw[directed, out=180, in=90] (7.center) to (2.center);
\end{tikzpicture}
     &
     \begin{tikzpicture}[scale=0.2]
		\node [] (0) at (-1, -1) {};
		\node [] (1) at (-3, -1) {};
		\node [] (2) at (-5, -1) {};
		\node [] (3) at (1, -1) {};
		\node [] (4) at (3, -1) {};
		\node [] (5) at (5, -1) {};
		\node [] (6) at (2, 0) {};
		\node [] (7) at (-3, 1) {};
		\node [] (8) at (-2, 0) {};
		\node [] (9) at (3, 1) {};
		\draw[directed] (6.center) to (3.center);
		\draw[directed] (6.center) to (4.center);
		\draw[out=180, in=80, directed] (7.center) to (2.center);
		\draw[out=0, in=130] (7.center) to (6.center);
		\draw[directed] (8.center) to (0.center);
		\draw[directed] (8.center) to (1.center);
		\draw[out=180,in=50] (9.center) to (8.center);
		\draw[directed, out=0, in=80] (9.center) to (5.center);
\end{tikzpicture}
     \\
    \hline
\end{tabular}
\end{center}
Note that $v_0=m_0$ and $v_1=m_1$. We include the following two additional diagrams, and therefore exhaust all possibilities for fork diagrams with $6$ bottom boundary points:
\begin{align*}
    u_2=s_4.m_2=
    \begin{tikzpicture}[scale=0.2,baseline=0]
		\node [] (0) at (-1, -1) {};
		\node [] (1) at (-3, -1) {};
		\node [] (2) at (-5, -1) {};
		\node [] (3) at (1, -1) {};
		\node [] (4) at (3, -1) {};
		\node [] (5) at (5, -1) {};
		\node [] (6) at (-1, 0) {};
		\node [] (7) at (2, 1) {};
		\node [] (8) at (-3, 0.5) {};
		\draw[directed, out=0, in=90] (7.center) to (5.center);
		\draw[reverse directed] (3.center) to (7.center);
		\draw[directed] (6.center) to (0.center);
		\draw[directed, out=180, in=90] (6.center) to (1.center);
		\draw[reverse directed, in=0, out=90] (4.center) to (6.center);
		\draw[directed, out=180, in=90] (7.center) to (2.center);
\end{tikzpicture}
    , \hspace{0.5 in} u_3=s_2.m_3=
    \begin{tikzpicture}[yscale=0.2,xscale=-0.2,baseline=0]
		\node [] (0) at (-1, -1) {};
		\node [] (1) at (-3, -1) {};
		\node [] (2) at (-5, -1) {};
		\node [] (3) at (1, -1) {};
		\node [] (4) at (3, -1) {};
		\node [] (5) at (5, -1) {};
		\node [] (6) at (-1, 0) {};
		\node [] (7) at (2, 1) {};
		\node [] (8) at (-3, 0.5) {};
		\draw[directed, out=0, in=90] (7.center) to (5.center);
		\draw[reverse directed] (3.center) to (7.center);
		\draw[directed] (6.center) to (0.center);
		\draw[directed, out=180, in=90] (6.center) to (1.center);
		\draw[reverse directed, in=0, out=90] (4.center) to (6.center);
		\draw[directed, out=180, in=90] (7.center) to (2.center);
\end{tikzpicture}.
\end{align*}
\end{example}

\begin{proposition}\label{prop:resolvecrossing}
All fork diagram with $6$ boundary points are linear combinations of M-diagrams:
\begin{align*}
    &v_2=s_2.m_1=m_2+m_1-m_0, \hspace{0.5 in} v_3=s_4.m_1=m_3+m_1-m_0,\\
    &u_2=s_4.m_2=m_4+m_2-m_0, \hspace{0.5 in} u_3=s_2.m_3=m_4+m_3-m_0,\\
    &v_4=m_4+m_3+m_2+m_1-m_0.
\end{align*}
\end{proposition}
\begin{proof}
The first four identities follow similar computations. We illustrate one of them as an example.
\begin{align*}
   & u_2= \begin{tikzpicture}[scale=0.3,baseline=0]
		\node [] (0) at (-1, -1) {};
		\node [] (1) at (-3, -1) {};
		\node [] (9) at (3, 2.5) {};
		\node [] (2) at (-5, -1) {};
		\node [] (3) at (1, -1) {};
		\node [] (4) at (3, -1) {};
		\node [] (5) at (5, -1) {};
		\node [] (6) at (-1, 0) {};
		\node [] (7) at (2, 1) {};
		\node [] (8) at (-3, 0.5) {};
		\draw[directed, out=0, in=90] (7.center) to (5.center);
		\draw[reverse directed] (3.center) to (7.center);
		\draw[directed] (6.center) to (0.center);
		\draw[directed, out=180, in=90] (6.center) to (1.center);
		\draw[directed, in=180] (6.center) to (9.center);
		\draw[out=0] (9.center) to (4.center);
		\draw[directed, out=180, in=90] (7.center) to (2.center);
\end{tikzpicture}=\begin{tikzpicture}[scale=0.3,baseline=0]
		\node  (0) at (-1, -1) {};
		\node  (1) at (-3, -1) {};
		\node  (9) at (3, 2.5) {};
		\node  (2) at (-5, -1) {};
		\node  (3) at (1, -1) {};
		\node  (4) at (3, -1) {};
		\node  (5) at (5, -1) {};
		\node  (6) at (-1, 0) {};
		\node  (7) at (2, 1) {};
		\node  (8) at (-3, 0.5) {};
		\node  (10) at (3, 1) {};
		\node  (11) at (3.75, 0) {};
		\draw [reverse directed] (3.center) to (7.center);
		\draw [directed] (6.center) to (0.center);
		\draw [directed, in=90, out=180] (6.center) to (1.center);
		\draw [directed, in=180] (6.center) to (9.center);
		\draw [directed, in=90, out=180] (7.center) to (2.center);
		\draw [out=0] (9.center) to (10.center);
		\draw[directed] (7.center) to (10.center);
		\draw (10.center) to (11.center);
		\draw[directed] (11.center) to (4.center);
		\draw[directed] (11.center) to (5.center);
\end{tikzpicture}
+
\begin{tikzpicture}[scale=0.3,baseline=0]
		\node  (0) at (-1, -1) {};
		\node  (1) at (-3, -1) {};
		\node  (9) at (3, 2.5) {};
		\node  (2) at (-5, -1) {};
		\node  (3) at (1, -1) {};
		\node  (4) at (3, -1) {};
		\node  (5) at (5, -1) {};
		\node  (6) at (-1, 0) {};
		\node  (7) at (2, 1) {};
		\node  (8) at (-3, 0.5) {};
		\draw [reverse directed] (3.center) to (7.center);
		\draw [directed] (6.center) to (0.center);
		\draw [directed, in=90, out=180] (6.center) to (1.center);
		\draw [directed, in=180] (6.center) to (9.center);
		\draw [directed, in=90, out=180] (7.center) to (2.center);
		\draw [directed, out=0, in=90] (7.center) to (4.center);
		\draw [out=0] (9.center) to (5.center);
\end{tikzpicture}\\
&=\begin{tikzpicture}[scale=0.3,baseline=0]
		\node  (0) at (-1, -1) {};
		\node (1) at (-3, -1) {};
		\node  (9) at (3, 1.5) {};
		\node  (2) at (-5, -1) {};
		\node  (3) at (1, -1) {};
		\node  (4) at (3, -1) {};
		\node  (5) at (5, -1) {};
		\node  (6) at (-1, 0) {};
		\node  (7) at (3.75, 0) {};
		\node  (8) at (-3, 0.5) {};
		\node  (11) at (3.75, 0) {};
		\draw [directed] (6.center) to (0.center);
		\draw [directed, in=90, out=180] (6.center) to (1.center);
		\draw [directed, in=180] (6.center) to (9.center);
		\draw [directed, in=90, out=180] (7.center) to (2.center);
		\draw [directed] (11.center) to (4.center);
		\draw [directed] (11.center) to (5.center);
		\draw [out=0] (9.center) to (3.center);
\end{tikzpicture}
-\begin{tikzpicture}[scale=0.3,baseline=0]
		\node  (0) at (-1, -1) {};
		\node  (1) at (-3, -1) {};
		\node  (9) at (3, 2.5) {};
		\node  (2) at (-5, -1) {};
		\node  (3) at (1, -1) {};
		\node  (4) at (3, -1) {};
		\node  (5) at (5, -1) {};
		\node  (6) at (-1, 0) {};
		\node  (7) at (2, 1) {};
		\node  (8) at (-3, 0.5) {};
		\node  (11) at (3.75, 0) {};
		\draw [directed] (6.center) to (0.center);
		\draw [directed, in=90, out=180] (6.center) to (1.center);
		\draw [directed, in=180] (6.center) to (9.center);
		\draw [directed, in=90, out=180] (7.center) to (2.center);
		\draw [directed] (11.center) to (4.center);
		\draw [directed, out=0] (11.center) to (5.center);
		\draw [directed, out=180] (11.center) to (3.center);
		\draw [out=0, in=0]  (9.center) to (7.center);
\end{tikzpicture}
+
\begin{tikzpicture}[scale=0.3,baseline=0]
		\node  (0) at (-1, -1) {};
		\node  (1) at (-3, -1) {};
		\node  (9) at (3, 2.5) {};
		\node  (2) at (-5, -1) {};
		\node  (3) at (1, -1) {};
		\node  (4) at (3, -1) {};
		\node  (5) at (5, -1) {};
		\node  (6) at (-1, 0) {};
		\node  (7) at (2, 1) {};
		\node  (8) at (-3, 0.5) {};
		\draw [reverse directed] (3.center) to (7.center);
		\draw [directed] (6.center) to (0.center);
		\draw [directed, in=90, out=180] (6.center) to (1.center);
		\draw [directed, in=180] (6.center) to (9.center);
		\draw [directed, in=90, out=180] (7.center) to (2.center);
		\draw [directed, out=0, in=90] (7.center) to (4.center);
		\draw [out=0] (9.center) to (5.center);
\end{tikzpicture}=m_2-m_0+m_4.
\end{align*}
The last identity follows from the previous identities and the fact
\begin{align*}
    v_4&=s_4.v_2=s_4.(m_2+m_1-m_0)=(m_4+m_2-m_0)+(m_3+m_1-m_0)+m_0\\
    &=m_4+m_3+m_2+m_1-m_0.
\end{align*}
\end{proof}
\begin{example}
Following Example~\ref{mainexample},  
the transitioning matrix from the polytabloid basis $\mathcal{P}=\{T_0,T_1,T_2,T_3,T_4\}$ to the set of M-diagrams $\mathcal{M}=\{m_0,m_1,m_2,m_3,m_4\}$ as ordered set is uni-triangular.  In particular,  $\mathcal{M}$ constitutes a basis of $S^{(2,2,2)}$.
\end{example}
For a fork diagram $m$,  define  $c(m)$,   the number of crossings,  to be the number of pairs $\{(a,b),(c,d)\}$ such that one of the two cases hold
\begin{itemize}
\item  $(a,b)$ and $(c,d)$ are both left arcs in $m$,  $a<c<b<d$ or $c<a<d<b$. 
\item  $(a,b)$ and $(c,d)$ are both right arcs in $m$,  $a<c<b<d$ or $c<a<d<b$. 
\end{itemize}
We remark that the number of crossings define above is different from the actual number crossings one sees pictorially.  In particular,  crossings formed by a left arc and a right arc are not counted in the definition of $c(m)$ above.  Unless specified otherwise,  we will always refer to $c(m)$ when we mention the number of crossings.

Let $m$ be a fork diagram.  The boundary word associated to $m$ is obtained by first marking all of its left endpoints with $+$'s,  middle endpoints with $0$'s,  right endpoints with $-$'s,  then reading off the labels from left to right in order.

Let $w=(w_i)_{i=1}^{3n}\in \{+,0,-\}^{3n}$ be the boundary word of a fork diagram.  We impose the total order $-<0<+$ on the set of labels.  The number of inversions in $w$,  denoted as $\operatorname{Inv}(w)$,  is 
\begin{align*}
\operatorname{Inv}(w)=\# \{(i,j), 1\leq i<j\leq 3n \:|\:  w_i>w_j \}.  \label{eq:inversions}
\end{align*}

For two words $w_1,w_2$ in $\{+,0,-\}$,  we say $w_1\prec w_2$ if $w_1$ and $w_2$ differ at  exactly two (not necessarily adjacent) positions,  and $w_1$ is obtained from $w_2$ by one of the following: 
\begin{itemize}
\item Replacing a pair of $(+,0)$ with $(0,+)$,
\item Replacing a pair of $(+,-)$ with $(-,+)$,
\item Replacing a pair of $(0,-)$ with $(-,0)$.
\end{itemize} 

\begin{lemma}\label{lem:inv}
If $w\prec v$,  then $\operatorname{Inv}(w)< \operatorname{Inv}(v)$.
\end{lemma}
\begin{proof}
Suppose $w$ and $v$ differ at boundary point $a$ and $b$ for $1\leq a<b\leq 3n$.  The letters $w_i=v_i$ for $1\leq i<a$ and $b<i\leq 3n$ does not affect the count in $\operatorname{Inv}$ after swapping $v_a$ and $v_b$.  Also,  by looking at the letters between (and not including) positions $a$ and $b$,  the count among themselves remain constant before and after the swap.  Therefore we only consider a letter $v_t$ for $a<t<b$,  and form the pair against either position $a$ or position $b$.  We pick one of the three cases,   where $v_a=+$ and $v_b=0$,   and divide it further into three subcases.  The other scenarios follow a similar argument. 

1)  Suppose $v_t=+$.    Then $(t,b)$ contributes to the count in $\operatorname{Inv}$ for $v$ but not for $w$.  The value of $\operatorname{Iv}$ drops by one.

2) Suppose $v_t=0$,   it is symmetric to the above case and the value of $\operatorname{Iv}$ drops by one.

3) Suppose $v_t=-$,  then $(a,t)$ contributes to the count in $\operatorname{Inv}$ for $v$,  and $(a,t)$ contributes to the count in $\operatorname{Inv}$ for $w$.  The value of $\operatorname{Iv}$ remains constant.

The above argument shows that $w$ has a count smaller than or equal to the count of $v$.  Now we look at the pair $(a,b)$ itself,  which contributes to the count in $v$ but not $w$,  we have $\operatorname{In}(w)<\operatorname{In}(v)$.
\end{proof}

\begin{remark}
The set of boundary words for $\mathcal{P}^n$,  or equivalently for $\mathcal{M}^n$,  is in bijection with $\mathcal{T}^n$.  This bijection is defined by setting a vertex $i$ to be $+$ ($0$ and  $-$,  respectively) if it is in the first (second and third,  respectively) row of $T\in \mathcal{T}^n$.  One can take the transitive closure of $\preceq$ and define a new order on $\mathcal{T}^n$.  This is actually a partial order,  as Lemma~\ref{lem:inv} ensures the reflexiveness.   It is a refinement of the partial order $\leq$ on $\mathcal{T}^n$ defined at the beginning.  In fact,  $\leq$ can be defined by requiring the pairs of symbols to be adjacent in the definition of $\preceq$.  One can also show that the two orders are not the same.  For example,  the two tableaux $T=\begin{ytableau}
1 & 3 & 6\\
2 & 4 & 8\\
5 & 7 & 9
\end{ytableau}$ and  $S=\begin{ytableau}
1 & 3 & 4\\
2 & 6 & 8\\
5 & 7 & 9
\end{ytableau}$ are comparible in the transitive closure of $\preceq$ but not in $\leq$.  
\end{remark}

\begin{example}
We list the boundary word and number of crossings for all fork diagrams in Example~\ref{mainexample}.
\begin{center}
\begin{tabular}{|c|c|c|c|c|c|}
\hline
$m$ & $v_0=m_0$ & $v_1=m_1$ & $v_2$ & $v_3$ & $v_4$  \\
\hline
$w$ &  $+0-+0-$ & $+0+-0-$ &$++0-0-$&$+0+0--$& $++00--$\\
\hline
$c(m)$ & $0$ & $0$ & $1$ & $1$  & $2$ \\
\hline
\hline
$m$&$u_2$ & $u_3$ & $m_2$ & $m_3$ & $m_4$\\
\hline
$w$ & $w_4$ & $w_4$ &$w_2$&$w_3$& $w_4$\\
\hline
$c(m)$ & $1$ & $1$ & $0$ & $0$  & $0$\\
\hline
\end{tabular}
\end{center}

Here $w_i$ is the boundary word for $v_i$,  $0\leq i\leq 4$.   Also note $w_0\prec w_1\prec w_2\prec w_4$ and $w_0\prec w_1\prec w_3\prec w_4$.
\end{example}

\begin{theorem}\label{th:basis}
The M-diagrams $\mathcal{M}^{n}$ forms a basis for $S^{(n,n,n)}$.  
\end{theorem}
\begin{proof}
For spanning,  we show that any polytabloid diagram can be written as a linear combination of M-diagrams.  We induct on the lexicographical order,  on the number of inversions in the boundary word,  followed by the number of crossings $c(m)$.  The base case for $c(m)=0$ is straightforward as the two bases coincide.   Now suppose $m$ is a fork diagram with $c(m)>0$.  Choose a crossing (formed by two left arcs intersecting or two right arcs intersecting) in $m$ and focus on the two arcs it involves.  We claim one can use Proposition~\ref{prop:resolvecrossing}  to rewrite it into M-diagrams with either strictly less number of crossings,  or whose boundary word has strictly lower inversions.  Notice this process may interfere with other arcs,  but the number of inversions only depends on the boundary labels,  which is determined by whether an endpoint is a left,  middle,  or right endpoint,  and this depends purely on the 14 cases given in Example~\ref{mainexample}.  Notice also that in most cases,  i.e.  as long as one starts with $v_2$,  $v_3$,  $v_4$,  one gets a linear combination of diagrams with strictly less inversions on the boundary words,  and the induction follows.

We now consider one of the remaining two cases ($u_2$ and $u_3$) and the other is similar.   Suppose there is a fork diagram $m$ containing $u_2$ as a subdiagram,  and one uses Proposition~\ref{prop:resolvecrossing} to rewrite it  into fork diagrams containing $m_4$,  $m_2$,  $m_1$ as subdiagrams.  The diagram containing $m_4$ is denoted as $m'$,  where the others have  boundary words with strictly less number of inversions.  We compare the pictures of $m$ and $m'$,  with right arcs highlighted in red.
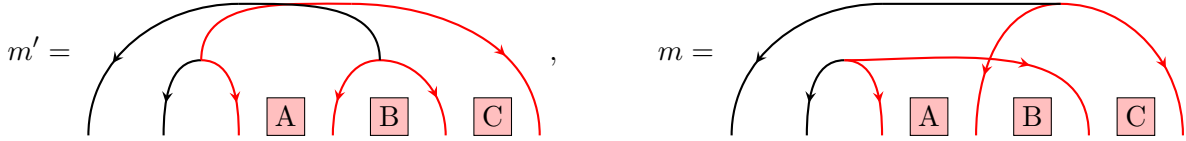
\begin{align*}
m'=
\begin{tikzpicture}[scale=0.5,  baseline=0]
		\node [] (0) at (-4, -2) {};
		\node [] (1) at (-2, -2) {};
		\node [] (2) at (0.5, -2) {};
		\node [] (3) at (3.5, -2) {};
		\node [] (4) at (1.75, 0) {};
		\node [] (5) at (6, -2) {};
		\node [] (6) at (-3, 0) {};
		\node [] (7) at (1, 1.5) {};
		\node [] (8) at (-2, 1.5) {};
		\node [] (9) at (-6, -2) {};
		\filldraw [fill=pink,  draw=black] (-1.25, -2)  rectangle ++ (1,1) ;
		\filldraw [fill=pink,  draw=black] (1.5, -2)  rectangle ++ (1,1) ;
		\filldraw [fill=pink,  draw=black] (4.25, -2)  rectangle ++ (1,1) ;
		\node (10) at (-0.75,-1.5) {A};
		\node (11) at (2,-1.5) {B};
		\node (12)at (4.75,-1.5) {C};
		\draw [directed,  out=180,  in=90,  thick] (6.center) to (0.center);
		\draw [directed,  out=0,  in=90,  red,  thick]  (6.center) to (1.center);
		\draw [out=90,  in=180,  red,  thick]  (6.center) to (7.center);
		\draw [directed,  out=0,  in=90,  red,  thick] (7.center) to (5.center);
		\draw [directed,  out=180,  in=90,  red,  thick]  (4.center) to (2.center);
		\draw  [directed,  out=0,  in=90,  red,  thick]  (4.center) to (3.center);
		\draw [out=0,  in=90,  thick] (8.center) to (4.center);
		\draw  [directed,  out=180,  in=90,  thick] (8.center) to (9.center);
\end{tikzpicture},\hspace{.5 in}
m=
\begin{tikzpicture}[scale=0.5,  baseline=0]
		\node [] (0) at (-4, -2) {};
		\node [] (1) at (-2, -2) {};
		\node [] (2) at (0.5, -2) {};
		\node [] (3) at (3.5, -2) {};
		\node [] (4) at (2.75, 1.5) {};
		\node [] (5) at (6, -2) {};
		\node [] (6) at (-3, 0) {};
		\node [] (8) at (-2, 1.5) {};
		\node [] (9) at (-6, -2) {};
		\filldraw [fill=pink,  draw=black] (-1.25, -2)  rectangle ++ (1,1) ;
		\filldraw [fill=pink,  draw=black] (1.5, -2)  rectangle ++ (1,1) ;
		\filldraw [fill=pink,  draw=black] (4.25, -2)  rectangle ++ (1,1) ;
			\node at (-0.75,-1.5) {A};
		\node at (2,-1.5) {B};
	\node at (4.75,-1.5) {C};
		\draw [directed,  out=180,  in=90,  thick] (6.center) to (0.center);
		\draw [directed,  out=0,  in=90,  red,  thick]  (6.center) to (1.center);
		\draw [directed,  out=180,  in=90,  red,  thick]  (4.center) to (2.center);
		\draw  [directed,  out=0,  in=90,  red,  thick]  (4.center) to (5.center);
		\draw [out=0,  in=180,  thick] (8.center) to (4.center);
		\draw [directed,  out=0,  in=90,  red,  thick] (6.center) to (3.center);
		\draw  [directed,  out=180,  in=90,  thick] (8.center) to (9.center);
\end{tikzpicture}
\end{align*}

We may assume that any two left arcs (or right arcs) intersect at most once due to Lemma~\ref{lem:R12}.  Notice the set of left arcs remains unchanged in $m$ and $m'$.  To count $c(m)$ we observe the remaining right arcs.  We first discuss the case when another right arc starts in region A and ends in region C.   This arc does not intersect the red arcs in $m'$ but intersects both red arcs in $m$.   The number of crossings it induces is bigger in  $m$ than that in $m'$.  In all other cases,  the number of crossings in $m$ and $m'$ results in the same count and is left to the reader.  Taking into account the red arcs shown in the picture,  which gives one more crossing in $m$ than $m'$,  we have shown that $c(m')<c(m)$.

Linear independence follows from the fact that $\mathcal{P}^n$ and $\mathcal{M}^n$ have the same cardinality which is equal to the dimension of $S^{(n^3)}$.
\end{proof}

This combined with Theorem~\ref{unitriangular} leads us to the following result:

\begin{corollary}\label{alsouni}
There is a unitriangular base change between any pair of the three bases: the polytabloid basis $\mathcal{P}^n$,  the non-elliptic basis $\mathcal{W}^n$,  and the M-diagram basis $\mathcal{M}^n$.
\end{corollary}

\bibliography{litlist} \label{references}
\bibliographystyle{amsalpha}

\end{document}